\documentclass[]{amsart}

\usepackage[english]{babel}
\usepackage[utf8]{inputenc}
\usepackage{paralist}
\usepackage{amsmath, amsfonts, amssymb, amsthm}
\usepackage{color}
\usepackage[normalem]{ulem}
\usepackage{tikz}

\newcommand{\QQ}{\mathbb{Q}}

\newcommand{\ZZ}{\mathbb{Z}}

\newcommand{\Oc}{\mathcal{O}}

\newtheorem{mymasterthm}{notForUse}%[section]
\theoremstyle{definition}

\newtheorem{myrem}[mymasterthm]{Remark}

\theoremstyle{plain}

\newtheorem{mythm}[mymasterthm]{Theorem}

\allowdisplaybreaks

\title[Separated Variables in Polynomial Power Sums]{Diophantine Equations in Separated Variables and Polynomial Power Sums}
\subjclass[2010]{11B37, 11C08}
\keywords{Diophantine equation, Bilu-Tichy theorem, linear recurrences}

\author[C. Fuchs]{Clemens Fuchs}
\author[S. Heintze]{Sebastian Heintze}
\thanks{Supported by Austrian Science Fund (FWF): I4406.}

\address{University of Salzburg\newline
	\indent Department of Mathematics\newline
	\indent Hellbrunnerstr. 34 \newline
	\indent A-5020 Salzburg, Austria}
\email{clemens.fuchs@sbg.ac.at, sebastian.heintze@sbg.ac.at}

\begin{document}
	
	\maketitle
	
%	{
%		\noindent
%		\textcolor{red}{working version \hfill \today}
%	}
	
	\begin{abstract}
		We consider Diophantine equations of the shape $ f(x) = g(y) $, where the polynomials $ f $ and $ g $ are elements of power sums.
		Using a finiteness criterion of Bilu and Tichy, we will prove that under suitable assumptions infinitely many rational solutions $ (x,y) $ with a bounded denominator are only possible in trivial cases.
	\end{abstract}
	
	\section{Introduction}
	
	Let $ f $ and $ g $ be integer polynomials. Diophantine equations of the shape $ f(x) = g(y) $ were already considered by many authors and under different assumptions. See \cite{kreso-tichy-2018} for an overview.
	
	Bilu and Tichy gave in \cite{bilu-tichy-2000} a criterion based on Siegel's theorem which characterizes the situations when the equation $ f(x) = g(y) $ has infinitely many rational solutions with a bounded denominator (see also \cite{bilu-fuchs-luca-pinter-2013}). For that, they used the notion of so-called \emph{standard pairs}. We shall describe standard pairs and their result in section \ref{p6-sec:prelim}.
	
	Furthermore, several authors studied the case when $ f $ and/or $ g $ come from special families of polynomials.
	Recently, Kreso (cf. \cite{kreso-2017}) considered the case when $ f $ and $ g $ are lacunary polynomials and used the criterion of Bilu and Tichy to deduce results about the finiteness of the number of solutions of the equation $ f(x) = g(y) $.
	Lacunary polynomials are polynomials of the shape $ c_1 x^{e_1} + \cdots + c_l x^{e_l} + c_{l+1} $ for a fixed number $ l $ of nonconstant terms where the $ c_i $ and $ e_i $ may vary with the only restriction that the $ e_i $ must be pairwise distinct.
	Kreso proved that under some assumptions on the exponents $ e_i $ and if $ g $ is indecomposable, then $ f(x) = g(y) $ has infinitely many solutions with a bounded denominator if and only if $ f = g \circ \mu $ for a linear polynomial $ \mu $.
	
	Dujella and Tichy proved in \cite{dujella-tichy-2001} the finiteness of the number of integer solutions for the situation when $ f,g $ are generalized Fibonacci polynomials.
	Moreover, Dujella and Gusic \cite{dujella-gusic-2007} as well as Stoll \cite{stoll-2008} considered families of polynomials parametrized by two parameters and a binary recurrence relation.
	Beyond this the case of truncated binomial polynomials was considered in \cite{dubickas-kreso-2016} by Dubickas and Kreso, sums of products of consecutive integers are considered in \cite{bazso-berczes-hajdu-luca-2018} by Bazso et al., and Bernoulli and Euler polynomial related families in \cite{pinter-rakaczki-2016} by Pinter and Rakaczki.
	
	In the present paper we are also considering Diophantine equations of the type $ f(x) = g(y) $. Here we are going to assume that the polynomials $ f $ and $ g $ come from polynomial power sums, i.e. simple linear recurrence sequences of polynomials.
	We remark that polynomial power sums can be seen as a variant of lacunary polynomials since its Binet representation has a fixed number of summands.
	
	\section{Results}
	
	Let $ G_n(x) = a_1 \alpha_1(x)^n + \cdots + a_d \alpha_d(x)^n $ with $ d \geq 2 $ and polynomial characteristic roots $ \alpha_1(x), \ldots, \alpha_d(x) \in \QQ[x] $ as well as constants $ a_1, \ldots, a_d \in \QQ $ be the $ n $-th polynomial in a linear recurrence sequence of polynomials satisfying the dominant root condition $ \deg \alpha_1 > \max_{i=2,\ldots,d} \deg \alpha_i $ and having at most one constant characteristic root. Assume furthermore that $ G_n(x) $ cannot be written in the form $ \widetilde{a_1} \widetilde{\alpha_1}(x)^n + \widetilde{a_2} \widetilde{\alpha_2}^n $ for $ \widetilde{\alpha_1}(x) \in \QQ[x] $ a perfect power of a linear polynomial and $ \widetilde{a_1}, \widetilde{a_2}, \widetilde{\alpha_2} \in \QQ $.
	We will refer to the assumptions in this paragraph by saying \emph{$ G_n(x) = a_1 \alpha_1(x)^n + \cdots + a_d \alpha_d(x)^n $ is the $ n $-th polynomial in a linear recurrence sequence of the required shape}.
		
	We call a polynomial $ f $ of degree $ \deg f \geq 2 $ decomposable if it can be written in the form $ f = g \circ h $ for polynomials $ g, h $ satisfying $ \deg g \geq 2 $ and $ \deg h \geq 2 $. Here $ \circ $ denotes the composition of functions. If such a decomposition does not exist, then we call the polynomial $ f $ indecomposable.
	
	Furthermore, we say that an equation $ f(x) = g(y) $ has infinitely many rational solutions with a bounded denominator if there exists a positive integer $ z $ such that $ f(x) = g(y) $ has infinitely many solutions $ (x,y) \in \QQ^2 $ with $ zx,zy \in \ZZ $.
	
	Our main result is the following theorem. In Remark \ref{p6-rem:generalize} below we give a possibility how to generalize it to arbitrary number fields:
	
	\begin{mythm}
		\label{p6-thm:mainthm}
		Let $ G_n(x) = a_1 \alpha_1(x)^n + \cdots + a_d \alpha_d(x)^n $ be the $ n $-th polynomial in a linear recurrence sequence of the required shape.
		Analogously, let $ H_m(y) = b_1 \beta_1(y)^m + \cdots + b_t \beta_t(y)^m $ be the $ m $-th polynomial in a linear recurrence sequence of the required shape.
		Moreover, assume that $ n,m > 2 $. If $ G_n(x) $ is indecomposable, then the equation in separated variables
		\begin{equation}
			\label{p6-eq:equation}
			G_n(x) = H_m(y)
		\end{equation}
		has infinitely many rational solutions with a bounded denominator if and only if there exists a polynomial $ P(y) \in \QQ[y] $ such that $ H_m(y) = G_n(P(y)) $ holds identically.
		
		If in addition $ H_m(y) $ is also indecomposable, then in the above statement we can restrict $ P(y) $ to be linear.
	\end{mythm}
	
	We exclude the case when $ G_n $ or $ H_m $ has exactly one constant and one nonconstant characteristic root, where the nonconstant one is (a perfect power of) a linear polynomial, since the conclusion is not true in general in this situation.
	Consider for instance $ G_n(x) = a(ex+c)^n+b $ and $ H_m(y) = a(fy+d)^m+b $ for integers $ a,b,c,d,e,f $, where $ a,e,f $ are non-zero, and different primes $ n,m $.
	Then all other assumptions of the theorem are satisfied.
	Moreover, there is no polynomial $ P $ such that $ H_m(y) = G_n(P(y)) $ since the degrees of $ G_n $ and $ H_m $ are different primes.
	But there are obviously infinitely many rational solutions with a bounded denominator of the form $ x = (t^m-c)/e $ and $ y = (t^n-d)/f $ for $ t \in \ZZ $.
	
	Now we give two examples where all assumptions of the theorem are satisfied and where in the first one we have infinitely many rational solutions with a bounded denominator whereas in the second one there are only finitely many such solutions. Thus both situations can occur.
	Let
	\begin{align*}
		G_3(x) &= (x^2)^3 + (x+1)^3 = x^6 + x^3 + 3x^2 + 3x + 1, \\
		H_3(y) &= (y^4-2y^2+1)^3 + (y^2)^3 = y^{12} - 6y^{10} + 15y^8 - 19y^6 + 15y^4 - 6y^2 + 1.
	\end{align*}
	We leave it up to the reader to verify that all assumptions of the theorem are satisfied. One can check that the identity $ H_m(y) = G_n(P(y)) $ holds for the polynomial $ P(y) = y^2-1 $. Therefore, by Theorem \ref{p6-thm:mainthm}, we have infinitely many rational solutions with a bounded denominator.
	If we consider $ G_3(x) $ from above and
	\begin{equation*}
		H_7(y) = (y^2)^7 + (y+2)^7,
	\end{equation*}
	then we get $ \deg G_3 = 6 $ as well as $ \deg H_7 = 14 $. Hence $ \deg G_3 $ does not divide $ \deg H_7 $ and therefore there is no polynomial $ P $ such that $ H_m(y) = G_n(P(y)) $. By Theorem \ref{p6-thm:mainthm} we cannot have infinitely many rational solutions with a bounded denominator.
	
	Note that we can check whether there exists a polynomial $ P(y) $ such that $ H_m(y) = G_n(P(y)) $ holds a priori. To do so we first determine $ \deg P $ by the equality $ \deg H_m = \deg G_n \cdot \deg P $. If this equation has no solution in positive integers, then there is no such polynomial $ P $. Otherwise we start with a polynomial $ P $ of the given degree and unknown coefficients. By a comparison of coefficients we determine step by step (starting with the leading coefficient) the values for the coefficients of $ P $. If we end up in a contradiction, then there is no such polynomial $ P $. In the case that there is no contradiction we have found a polynomial with the sought property.
	We remark that in the case that there are only finitely many solutions our result is ineffective in the sense that we do not find all the solutions (for a given common denominator).
	
	Note that $ G_n $ and $ H_m $ can be elements of different linear recurrence sequences, but they could also be elements of the same linear recurrence sequence. We do neither require nor exclude the situation $ G_n(x) = G_m(y) $ for $ n \neq m $ if the assumptions of our theorem are satisfied.
	
	Furthermore, we remark that the polynomial of the second linear recurrence sequence $ H_m $ can be replaced by an arbitrary fixed polynomial $ h(y) \in \QQ[y] $. If we replace all assumptions about $ H_m $ by the two assumptions that $ \deg h > 4 $ and that $ h $ is not of the shape $ h(y) = a(cy+d)^k+b $ for rational numbers $ a,b,c,d $, then the same result as in Theorem \ref{p6-thm:mainthm} holds. The proof is completely analogous to the below given proof of Theorem \ref{p6-thm:mainthm}.
	
	\section{Preliminaries}
	\label{p6-sec:prelim}
	
	The proof of our theorem uses a criterion of Bilu and Tichy \cite{bilu-tichy-2000}, for which the following terminology of so-called \emph{standard pairs} is needed.
	
	In our notation, $ k $ and $ l $ are positive integers, $ a $ and $ b $ are non-zero rational numbers and $ p(x) $ is a non-zero polynomial with coefficients in $ \QQ $.
	We denote by $ D_k(x,a) $ the $ k $-th Dickson polynomial which is defined by the equation
	\begin{equation*}
		D_k \left( x + \frac{a}{x}, a \right) = x^k + \left( \frac{a}{x} \right)^k.
	\end{equation*}
	Using this notation we have the following five kinds of \emph{standard pairs} (over $ \QQ $); in each of them the two coordinates can be switched: A standard pair of the
	\begin{itemize}
		\item \emph{first kind} is
		\begin{equation*}
			(x^k, a x^l p(x)^k)
		\end{equation*}
		with $ 0 \leq l < k $, $ \gcd(k,l) = 1 $ and $ l + \deg p(x) > 0 $;
		\item \emph{second kind} is
		\begin{equation*}
			(x^2, (ax^2+b) p(x)^2);
		\end{equation*}
		\item \emph{third kind} is
		\begin{equation*}
			(D_k(x,a^l), D_l(x,a^k))
		\end{equation*}
		with $ \gcd(k,l) = 1 $;
		\item \emph{fourth kind} is
		\begin{equation*}
			(a^{-k/2} D_k(x,a), -b^{-l/2} D_l(x,b)).
		\end{equation*}
		with $ \gcd(k,l) = 2 $;
		\item \emph{fifth kind} is
		\begin{equation*}
			((ax^2-1)^3, 3x^4-4x^3).
		\end{equation*}
	\end{itemize}
	
	Our main tool is now the following theorem which is proven as Theorem 1.1 in \cite{bilu-tichy-2000} by Bilu and Tichy:
	\begin{mythm}
		\label{p6-thm:btcriterion}
		Let $ f(x), g(x) \in \QQ[x] $ be non-constant polynomials. Then the following two assertions are equivalent:
		\begin{enumerate}[a)]
			\item The equation $ f(x) = g(y) $ has infinitely many rational solutions with a bounded denominator.
			\item We have $ f = \varphi \circ f_1 \circ \lambda $ and $ g = \varphi \circ g_1 \circ \mu $, where $ \lambda(x), \mu(x) \in \QQ[x] $ are linear polynomials, $ \varphi(x) \in \QQ[x] $, and $ (f_1(x), g_1(x)) $ is a standard pair over $ \QQ $ such that the equation $ f_1(x) = g_1(y) $ has infinitely many rational solutions with a bounded denominator.
		\end{enumerate}
	\end{mythm}
	
	\section{Proof}
	
	All necessary preparations that are needed for the proof of our theorem are finished. So we can start with the proof:
	
	\begin{proof}[Proof of Theorem \ref{p6-thm:mainthm}]
		First note that by the dominant root condition we have the bounds $ \deg \alpha_1 \geq 1 $ and $ \deg G_n = n \deg \alpha_1 > 2 $. Analogously, the bound $ \deg H_m = m \deg \beta_1 > 2 $ holds.
		
		The next important observation is that we can neither have $ \deg \alpha_1 = 1 $ nor $ \deg \beta_1 = 1 $. Otherwise, if $ \deg \alpha_1 = 1 $, then $ G_n(x) $ would have exactly two characteristic roots and one of them would be constant. This shape is forbidden by the conditions of the theorem. The argument for $ \deg \beta_1 $ is the same.
		
		Now assume that equation \eqref{p6-eq:equation} has infinitely many rational solutions with a bounded denominator.
		Thus, by Theorem \ref{p6-thm:btcriterion}, we have
		\begin{equation*}
			G_n = \varphi \circ g \circ \lambda
		\end{equation*}
		and
		\begin{equation*}
			H_m = \varphi \circ h \circ \mu
		\end{equation*}
		for a polynomial $ \varphi(x) \in \QQ[x] $, linear polynomials $ \lambda(x), \mu(x) \in \QQ[x] $ and a standard pair $ (g(x), h(x)) $.
		
		From here on we distinguish between two cases. In the first case we assume that $ \deg \varphi = 1 $.
		
		Then $ (g(x), h(x)) $ cannot be a standard pair of the first kind. Otherwise we would either have
		\begin{equation*}
			G_n(x) = e_1 (\lambda(x))^{n \deg \alpha_1} + e_0 = e_1 \left( (\lambda(x))^{\deg \alpha_1} \right)^n + e_0
		\end{equation*}
		or
		\begin{equation*}
			H_m(y) = e_1 (\mu(y))^{m \deg \beta_1} + e_0 = e_1 \left( (\mu(y))^{\deg \beta_1} \right)^m + e_0,
		\end{equation*}
		which contradicts the restrictions on the shape of $ G_n(x) $ and $ H_m(y) $.
		
		Moreover, $ (g(x), h(x)) $ cannot be a standard pair of the second kind since we have $ \deg G_n > 2 $ and $ \deg H_m > 2 $.
		
		If $ (g(x), h(x)) $ is a standard pair of the third kind, then we get
		\begin{equation}
			\label{p6-eq:gndickson}
				G_n(x) = e_1 D_p(\lambda(x),a) + e_0.
		\end{equation}
		Since $ G_n(x) $ is indecomposable and Dickson polynomials have the composition property
		\begin{equation*}
			D_{kl}(x,a) = D_k(D_l(x,a),a^l)
		\end{equation*}
		the index $ p $ in \eqref{p6-eq:gndickson} must be a prime.
		Hence
		\begin{equation*}
			n \deg \alpha_1 = \deg G_n = \deg D_p = p
		\end{equation*}
		together with $ n > 2 $ implies $ \deg \alpha_1 = 1 $. As shown above this is a contradiction.
		Therefore $ (g(x), h(x)) $ cannot be a standard pair of the third kind.
		
		Also, $ (g(x), h(x)) $ cannot be a standard pair of the fourth kind since otherwise
		\begin{equation*}
			G_n(x) = e_1 D_k(\lambda(x),a) + e_0
		\end{equation*}
		with an even $ k $ would contradict the fact that $ G_n(x) $ is indecomposable.
		
		Furthermore, $ (g(x), h(x)) $ cannot be a standard pair of the fifth kind. Otherwise we would have either $ g(x) = 3x^4-4x^3 $ or $ h(x) = 3x^4-4x^3 $. This means $ n \mid \deg G_n = 4 $ or $ m \mid \deg H_m = 4 $ and therefore $ n = 4 $ or $ m = 4 $, since $ n,m > 2 $. This ends up in the contradiction $ \deg \alpha_1 = 1 $ or $ \deg \beta_1 = 1 $.
		
		Thus the case $ \deg \varphi = 1 $ is not possible. So we can assume the second case, namely that $ \deg \varphi > 1 $.
		Since $ G_n $ is indecomposable, we have $ \deg g = 1 $. Consequently the identities
		\begin{equation*}
			G_n(x) = \varphi(c_1 x + c_0)
		\end{equation*}
		and
		\begin{equation*}
			H_m(y) = \varphi(q(y))
		\end{equation*}
		hold for a polynomial $ q(y) \in \QQ[y] $.
		Now we define the polynomial $ P(y) \in \QQ[y] $ by the equation
		\begin{equation*}
			P(y) := \frac{q(y)-c_0}{c_1}
		\end{equation*}
		which gives us the final identity
		\begin{equation*}
			G_n(P(y)) = G_n \left( \frac{q(y)-c_0}{c_1} \right) = \varphi (q(y)) = H_m(y).
		\end{equation*}
		
		If $ H_m(y) $ is indecomposable, then $ q(y) $ is linear. Thus by construction $ P(y) $ is linear, too.
		
		Conversely, if we assume the identity $ G_n(P(y)) = H_m(y) $, then equation \eqref{p6-eq:equation} obviously has infinitely many rational solutions with a bounded denominator.
	\end{proof}
	
	\begin{myrem}
		\label{p6-rem:generalize}
		We remark that if we utilize Theorem 10.5 in \cite{bilu-tichy-2000} instead of Theorem 1.1, then we can replace $ \QQ $ by an arbitrary number field $ K $ and get for a finite set $ S $ of places of $ K $, containing all archimedean ones, the analogous result as above for infinitely many solutions with a bounded $ \Oc_S $-denominator.
	\end{myrem}

\end{document}